\documentclass[10pt]{article}
\usepackage{amsmath,amssymb,amsthm}
\newtheorem{thm}{Theorem}
\newtheorem*{prop}{Proposition}
\newtheorem{lem}{Lemma}
\theoremstyle{definition}
\newtheorem{example}{Example}
\newtheorem{defn}{Definition}

\renewcommand{\i}{\iota}
\def\Re{{\rm Re}\,}
\def\Im{{\rm Im}\,}

\begin{document}

\title{On factorization of elements in Pimenov algebras}
\author{Dmitriy Efimov\thanks{e-mail: defimov@dm.komisc.ru}\\ Departmant of Mathematics,\\ Komi Science Centre UrD RAS,\\
Syktyvkar, Russia}
\date{}
%--------\wplace{Отдел математики, Коми научный центр УрО РАН, г.Сыктывкар 
\maketitle

\begin{abstract}
We consider the operation of division in Pimenov algebras.
We obtain necessary and sufficient conditions  for prime elements in Pimenov algebras
with a number of generators less than $5$.
We adduce examples of the factorization of elements in these algebras.
\end{abstract}

\section*{Introduction} 
A commutative associative algebra with unit generated by the finite number
of nilpotent of index $2$ generators arises in many areas of mathematics and theoretical physics.
R.I. Pimenov used it for a unified description of all $3^n$ Cayley-Klein geometries of dimension $n$ 
(geometries of spaces with constant curvature) \cite{pim2}.
Recently, its use is mainly due to the method of algebraic contractions \cite{grom} -- \cite{grom2}.
Pimenov algebra with one generator is the set of {\it dual numbers}, which
introduced W. K. Clifford in the second half of the XIX century.
These hypercomplex numbers are widely used in various fields of
mathematics and theoretical physics \cite{diment} -- \cite{sat2}.
In this work our attention will be focused on Pimenov algebras generated 
by more than one generator.  

We shall give a rigorous definition.
Let $K$ denote an arbitrary field of characteristic zero, for example,
the field of real or complex numbers.
\begin{defn}
{\it Pimenov algebra } with $n$ generators over $K$ is
an associative algebra generated over $K$ by unit and elements $\i_k,\ k=1,\dots, n$
with the defining relations
\begin{equation}\label{sootn}
\i_k^2=0,\ \i_k\i_l=\i_l\i_k,\ \  k,l=1,\dots, n.
\end{equation}
We shall denote it by $P_n(\i)$.
\end{defn}
From the definition it follows that the algebra $P_n(\i)$ is commutative,
has a unit, and each its element is  represented uniquely in the following {\itshape standard} form:
\begin{eqnarray}\label{1_1}
p=p_0+\sum_{t=1}^n\sum_{k_1<\dots<k_t} p_{k_1\dots k_t}\i_{k_1}\dots\i_{k_t},
\end{eqnarray}
where $p_0, p_{k_1\dots k_t}\in K$.
By analogy with complex numbers, the element $p_0$ we will call
{\itshape the real part} of the element $p$ and denote by $\Re p$, 
and the element $p-p_0$ we will call {\itshape the imaginary part} of the element $p$ and denote by $\Im p$.
Pimenov algebra with $n$ generators can be considered \cite{grom} as a subalgebra
of the even part of Grassmann algebra with $2n$ generators.
Recall that generators of the latter, in contrast to generators of Pimenov algebra,
are anticommuting \cite{grassmann}.

For integral rings, i.e. commutative associative rings
without zero divisors, one of the important issues is
the problem of the factorization of elements or, in other words,
the problem of their decomposition in the product of prime elements. 
{\it The prime element} is a nonzero not invertible element
that can not be represented in the form of a product 
of not invertible elements. Thus in the ring of integers
prime elements are the prime numbers,
in the ring of polynomials in one variable  they are the irreducible polynomials. 
For this rings it is known that each  nonzero  not invertible element
is represented uniquely up to permutations of factors and up to multiplication by invertible elements
as a product of  prime elements. 

The algebra $P_n(\i)$ is not an integral ring, it has zero divisors. 
Nevertheless the issue  of a factorization its elements 
makes also certain sence. For example, in some studies in theoretical physics using
Pimenov algebra  arise  expressions of the form $\iota_k/\iota_k$ \cite{grom_eng}, \cite{grom2}. 
Despite the fact that a division of the generators of Pimenov algebra is undefined
(as will be discussed in more detail in the next section), such expressions, which is quite natural,
are set equal to one.
From a mathematical point of view this is a special case of the broader issue
about a solution of the equation $ax=b$ with respect to $x$, 
where $a$ and $b$  are  not invertible elements of Pimenov algebra.
This equation has a solution only in the case, when $b$ may be decomposed into factors, one of which is $a$.

In the first section we will consider invertible elements in Pimenov algebra. 
In the second section we will examine direct  factorization of elements. 

\section{Division  in the algebra $P_n(\i)$}
%---------------------------------- Definition 2 ----------------------------
This section we will begin with some definitions.
\begin{defn}
{\itshape The length of a monomial} $\lambda\i_{k_1}\dots\i_{k_t}$ is the number $t$ of generators
of the algebra $P_n(\i)$ included in this monomial. By definition we will assume
that all elemets of the form $\lambda 1$, $\lambda\in K$ have  zero length.
\end{defn}
%-----------------------------------------------------------------------------
%--------------------------------- Definition 3 -----------------------------
\begin{defn}
{\itshape The degree} of an element $p\in P_n(\i)$ is the minimum length of monomials
included in the standart form (\ref{1_1}) of the element $p$.
The degree of an element with nonzero real part we will take equal $0$. 
The degree of zero element we shall assume equal $+\infty$.
An element $p\in P_n(\i)$ is called {\it homogeneous of degree $t,\ t\leq n$}
if its standart form  involve only monomial of the length $t$.
\end{defn}
%--------------------------------------------------------------------------------
It is quite obvious that for any pair of elements $a,b$ from the algebra $P_n(\i)$ we have 
\begin{equation}\label{1_5}
\deg ab\geq \deg a+\deg b.
\end{equation}
 
An element $p\in P_n(\i)$ is called {\it invertible}, if there exist an element 
$p^{-1}\in P_n(\i)$ such that $pp^{-1}=p^{-1}p=1$. 
The following theorem holds.
%--------------------------------- Theorem 1 -------------------------------------
\begin{thm}\label{1_4}
An element $p\in P_n(\i)$ is invertible if and only if its real part is zero.
If $p$ is invertible, then $p^{-1}$ is uniquely determined by the following formula:
\begin{equation}\label{obr}
 p^{-1}=\frac{1}{\Re p}\sum_{m=0}^M\left(-\frac{\Im p}{\Re p}\right)^m,
\end{equation}
where $M$ is a maximum degree such that $(\Im p)^M$ is not equal zero.
\end{thm}
\begin{proof}
A detailed proof of this theorem for  Grassmann algebra is given in \cite{grassmann}. 
It is true respectively  for Pimenov algebra as a subalgebra of Grassmann algebra.
Note that in \cite{witt} the right part of the formula (\ref{obr})
is given in the form of an infinite power series.
\end{proof}
%-----------------------------------------------------------------------------------
Taking into account the structure of Pimenov algebra
give here a more precise estimate for M compared with \cite{grassmann}.
Consider an arbitrary element $p\in P_n(\i)$. 
Let the number of monomials contained in the standard form of $\Im p$ is equal to $\mu$. 
Then the maximum degree, by raising to which $\Im p$ is not zero,
can be estimated as follows: 
\begin{equation}\label{max}
M\leq{\rm min}\left\{\left[\frac{n}{\deg(\Im p)}\right], \mu\right\},
\end{equation}
where $[\ ]$ denotes the integer part.

%-------------------------------------- Example 1 --------------------------------
\begin{example}
If $p=2+\i_1-\i_2\i_3$, then $M=2$ and 
\begin{equation*}
 \begin{split}
  p^{-1}&=\frac{1}{2}\left(1-\frac{\i_1-\i_2\i_3}{2}+\left(\frac{\i_1-\i_2\i_3}{2}\right)^2\right)=\\
  &=\frac{1}{2}-\frac{1}{4}\i_1+\frac{1}{4}\i_2\i_3-\frac{1}{8}\i_1\i_2\i_3.
 \end{split}
\end{equation*}
\end{example}
%-------------------------------------------------------------------------------------
From this example it follows that a number of monomials in an element and its inverse may be not coincide.
But it is easy to see that degrees of imaginary parts of $p$ and $p^{-1}$   
are coincide for any invertible element $p\in P_n(\i)$:
\begin{equation}
\deg(\Im p)=\deg(\Im p^{-1}).
\end{equation}

Now let us express in an explicit form some coefficients of monomials in an inverse element
through coefficients of monomials in an element itself.
Assume that we are looking an inverse element $b$ of an element $a$.
Consider the equality:
$$
(a_0+\dots+a_{1\dots n}\i_1\dots\i_n)
(b_0+\dots+b_{1\dots n}\i_1\dots\i_n)=1
$$
Two elements of the algebra $P_n(\i)$ are equel if and only if
coefficients in their standart form $(\ref{1_1})$ before appropriate 
monomials are equal.
Multiplied brackets in the left part of the above equality  
and equeted coefficients of appropriate monomials
$\i_{k_1}\dots\i_{k_t}$ we get a system of $2^n$ linear equations with
$2^n$ unknowns. 
By multiplying of elements from $P_n(\i)$
their real parts are multiplied, i.e. $\Re ab=\Re a\Re b$.
If $a_0$ is not equal $0$ then it is invertible and from the equality $a_0b_0=1$
we get uniquely: 
\begin{equation}
 b_0=\frac{1}{a_0}\ .
\end{equation} 
Further, a generator $\i_k$ will be have the coefficient $c_k=a_0b_k+a_kb_0$ 
in the left part of the equality. 
From the equation $c_k=0$ we obtain uniquely:
\begin{equation}
 b_k=-\frac{a_kb_0}{a_0}=-\frac{a_k}{a^2_0},\ \ \ k=1,2,\dots,n.
\end{equation}
In the next step we find coefficients $b_{k_1k_2}$ with double-digit lower indexes.
It is easy to see that each of them is uniquely expressed as a fraction, 
the denominator of which is $a_0$, and the numerator is  an expression of
the coefficients have already been found:
\begin{equation}
\begin{split}
 b_{k_1k_2}=-&\frac{a_{k_1}b_{k_2}+a_{k_2}b_{k_1}+a_{k_1k_2}b_0}{a_0}=\\
 &=\frac{2a_{k_1}a_{k_2}-a_0a_{k_1k_2}}{a_0^3}.
\end{split}
\end{equation}
Similarly, using the already calculated coefficients, we find coefficients with three lower indices.
\begin{equation}
 \begin{split}
  b_{k_1k_2k_3}=&\frac{2a_0(a_{k_1}a_{k_2k_3}+a_{k_2}a_{k_1k_3}+a_{k_3}a_{k_1k_2})}{a_0^4}\,-\\
  &-\frac{6a_{k_1}a_{k_2}a_{k_3}+a_0^2a_{k_1k_2k_3}}{a_0^4}.
 \end{split}
\end{equation}

By definition, the problem of division in the algebra $P_n(\i)$ is equivalent to
the question about solution of the equation
\begin{eqnarray}\label{1_3}
ax=b,\ \ \ a,b\in P_n(\i).
\end{eqnarray}

\begin{prop}\label{1_6}
If an element $a$ is invertible, then the equation (\ref{1_3})
has the unique solution $x=a^{-1}b$. If $b$ is invertible,
$a$ is not invertible, then (\ref{1_3}) does not have solutions. 
If $a$ and $b$ are not invertible, at the same time  non-zero elements, 
then the solution can exist and not exist.
If a solution exists, then it may be not unique, 
and all solutions are invertible or all solutions are not invertible.
\end{prop}

\begin{proof}
Let $ac=b$ and $ac_1=b$, where $a$ is invertible.
Then it follows, that $c=c_1=a^{-1}b$. 

The real part of the product of two elements is equal to the product of their real parts.
Therefore, if $a$ is not invertible, then the real part of the element $ax$ will be zero
and the equation $ax=b$, where $b$ is invertible, does not have solutions.

Let $a$ and $b$ be not invertible.
If $x$ is a solution of the equation $ax=b$, then
$x+\lambda\i_1\i_2\dots\i_n$, $\lambda\in K$ is also
a solution of this equation. If $x$ is invertible, then
$\deg a=\deg b$. And if $x$ is not invertible, then $\deg b>\deg a$.
It follows, that either all solutions of the equation $ax=b$ are invertible,
or all of them are not invertible. 

The equation $\i_2x=\i_2+\i_1\i_2$ has the solution $x=1+\i_1$. 
The equation $\i_1\i_2x=\i_1$ does not have solutions.  
Indeed, if there exist a solution, 
then by (\ref{1_5}) the inequality $\deg\i_1\geq \deg{\i_1\i_2}+\deg x$ is hold. 
But it is  not true obviously.
\end{proof}

Solvability of Equation (\ref{1_3}) in the case when $a$ and $b$ 
are not invertible is closely linked
with a prime factorization of not invertible elemets of Pimenov algebra. 
This problem will be discussed in the next section.

\section{Factorization}
\begin{defn}
By analogy with the integral rings,
a nonzero invertible element, which can not
be represented as a product of non-invertible 
elements in the algebra $P_n(\i)$
is said to be {\it a prime element}.
\end {defn}
Using inequality (\ref{1_5}) we can easily show 
that each nonzero not-invertible element in the algebra $P_n(\i)$ 
is a product of a finite number of prime elements.
Such product can not be unique.
For example, $(\i_1+\i_2)\cdot\i_2=\i_1\cdot\i_2$  are two different ways
of factorization of the element $\i_1\i_2$.

It follows from (\ref{1_5}), that 
all elements of the first degree 
are prime elements in $P_n(\i)$.  
Consider the question, does $P_n(\i)$ have 
prime elements other than elements of the first degree?
%------------------------------------- Теорема 3 ---------------------------------------------
\begin{thm}\label{t2}
In Pimenov algebras $P_2(\i)$ and $P_3(\i)$ prime elements
are only elements of the first degree.
\end{thm}

\begin{proof}
If an element in $P_2(\i)$ is not-invertible and its degree is different from $1$, then
it is a monomial $p=\lambda \i_1\i_2$, which, obviously, is not prime.

If an element in $P_3(\i)$ is not-invertible and its degree is more than $1$, then it has the form:
\begin{equation}\label{p23}
 p=\alpha\i_1\i_2+\beta\i_1\i_3+\gamma\i_2\i_3+\delta\i_1\i_2\i_3,\ \ \ \alpha,\beta,\gamma,\delta\in K.
\end{equation}
Let us first consider its homogeneous part of the second degree, that is the element
\begin{equation*}\label{p24}
 f=\alpha\i_1\i_2+\beta\i_1\i_3+\gamma\i_2\i_3,\ \ \ \alpha,\beta,\gamma\in K.
\end{equation*}
Obviously, if $f$ is decomposible, it can be written as a product of
homogeneous elements of the first degree:
\begin{equation}\label{p25}
 f=(a_1\i_1+a_2\i_2+a_3\i_3)(b_1\i_1+b_2\i_2+b_3\i_3).
\end{equation} 
One of the coefficients, say $a_1$, we can choose to be $1$.
Expanding in (\ref{p25}) brackets, and equating coefficients, we obtain the system:
\begin{equation*}
\left\{
\begin{aligned}
b_2+a_2b_1=\alpha,\\
b_3+a_3b_1=\beta,\\
a_2b_3+a_3b_2=\gamma.
\end{aligned}
\right.
\end{equation*}
Taking $a_3 = a$ and $b_1 = b$ for the parameters, and expressing through them  the other coefficients, we get:
\begin{equation*}
a_2=\frac{\gamma-a\alpha}{\beta-2ab},\ \ b_2=\frac{\alpha\beta-b\gamma-ab\alpha}{\beta-2ab},\ \ b_3=\beta-ab.
\end{equation*}
Thus, the element $f$ is decomposible.
From this decomposition is easy to get and a decomposition of $p$ adding, for example,
to the second factor the term $\delta\i_2\i_3$.
So any element in $P_3(\i)$, the degree of which more than $1$, is not prime.
\end{proof}
%----------------------------------------------------------------------------------------------------

\begin{example}
If $\beta\not=0$, then assuming $a=b=0$ we get:
 \begin{equation*}\label{p22}
   \begin{split}
     \alpha\i_1\i_2+\beta\i_1\i_3+\gamma\i_2\i_3&+\delta\i_1\i_2\i_3=\\
     &=(\i_1+\frac{\gamma}{\beta}\i_2)(\alpha\i_2+\beta\i_3+\delta\i_2\i_3).
   \end{split} 
 \end{equation*}
\end{example}

%---------------------------------------- Лемма 1 ---------------------------------------------
The first example of Pimenov algebra, in which there exist nontrivial prime elemets, is the algebra $P_4(\i)$.
\begin{lem}\label{pre0}
Only elements of the following two types are prime among homogeneous elements of $2$-degree in  $P_4(\i)$:
\begin{enumerate}
\item $\alpha\i_a\i_b+\beta\i_a\i_c+\gamma\i_b\i_d,\ \ \alpha, \beta, \gamma\not=0$;  
\item $\alpha\i_a\i_b+\beta\i_a\i_c+\gamma\i_b\i_d+\delta\i_c\i_d,\ \ \alpha\beta\gamma\delta<0$;
\end{enumerate}
where all indexes $a,b,c,d$ are distinct and take value from  $1$ to $4$.
\end{lem}
\begin{proof}
A homogeneous element of $2$-degree in $P_4(\i)$ has the following general form:
\begin{equation}\label{p28}
p=\alpha\i_1\i_2+\beta\i_1\i_3+\gamma\i_1\i_4+\delta\i_2\i_3+\rho\i_2\i_4+\sigma\i_3\i_4.
\end{equation}
Obviously, that if $p$ is decomposable, then it can be written
as a product of homogeneous elements of $1$-degree:
\begin{equation}\label{p8}
p=\left(\sum_{s=1}^4a_s\i_s\right)\left(\sum_{t=1}^4b_t\i_t\right).
\end{equation} 
Expanding the brackets in this expression and equating coefficients, we obtain the system:
\begin{equation}\label{p9}
\left\{
\begin{aligned}
a_1b_2+a_2b_1=\alpha,\\
a_1b_3+a_3b_1=\beta,\\
a_1b_4+a_4b_1=\gamma,\\
a_2b_3+a_3b_2=\delta,\\
a_2b_4+a_4b_2=\rho,\\
a_3b_4+a_4b_3=\sigma.
\end{aligned}
\right.
\end{equation}

Homogeneous elements of $2$-degree consisting of two monomials can be divided into two classes.
The first class is the set of elements, in which both monomials contain one and the same generator.
Such elements are obviously decomposible.

The second class is the set of elements, in which each generator present
exactly in one monomial, i.e. elements of the form:
$$
p=\alpha\i_1\i_2+\sigma\i_3\i_4.
$$
Obviously, if an element is repsesented in the form (\ref{p8}), then one of the coefficients, say $a_1$,
can be taken equal to $1$. Then, taking the coefficients $a_3=a$ and $b_1=b$ for the parameters,
we find from $(\ref{p9})$ for the element $p$:
\begin{equation*}
a_2=\frac{\alpha}{2b},\ \ a_4=\frac{\sigma}{-2ab},\ \ b_2=\frac{\alpha}{2},\ \ b_3=-ab,\ \ b_4=\frac{\sigma}{2a}.
\end{equation*}
Thus the element
$p=\alpha\i_1\i_2+\sigma\i_3\i_4$ is decomposible.

Homogeneous elements of $2$-degree consisting of three monomials can be divided into three classes.
The first class is the set of elements all monomials which contain one and the same generator.
These elements are clearly decomposable.

The second class is the set of elements that have one and the same generator exactly in two monomials,
i.e. elements of the form:
\begin{equation*}\label{p11}
p=\alpha\i_1\i_2+\beta\i_1\i_3+\delta\i_2\i_3. 
\end{equation*}
By Theorem \ref{t2}, they are also decomposable.

The third class is the set of elements in which two generators appear twice,
and two at once, i.e. elements of the form
\begin{equation}\label{p14}
p=\alpha\i_1\i_2+\beta\i_1\i_3+\rho\i_2\i_4. 
\end{equation}
Under the assumption that $a_3=1$, the system (\ref{p9}) will take for them to form:
\begin{equation}\label{p15}
\left\{
\begin{aligned}
a_1b_2+a_2b_1=\alpha,\\
a_1b_3+b_1=\beta,\\
a_1b_4+a_4b_1=0,\\
a_2b_3+b_2=0,\\
a_2b_4+a_4b_2=\rho,\\
b_4+a_4b_3=0.
\end{aligned}
\right.
\end{equation}
From $2$, $4$ and $6$ equations we find:
\begin{equation}\label{p16}
 b_1=\beta-a_1b_3,\ \ b_2=-a_2b_3,\ \ b_4=-a_4b_3.
\end{equation} 
Substituting these expressions in the $5$-th equation, we get $-2a_2a_4b_3=\rho$. 
Since $\rho\not=0$, then it follows that $a_4\not=0$.
Substituting (\ref{p16}) in the equation $3$, we get $a_4(\beta-2a_1b_3)=0$.
Since $a_4\not=0$, then $\beta-2a_1b_3=0$. 
Substituting (\ref{p16}) in the  equation $1$, we obtain $a_2(\beta-2a_1b_3)=\alpha$. 
It follows that $\alpha=0$. We obtain a contradiction. 
Consequently, the system (\ref{p15}) is not consistent
and the element (\ref{p14}) is prime.

Now consider homogeneous elements of $2$-degree, consisting of four or more monomials.
Note that if in these elements any generator contains exactly in three monomials, these elements are not prime. 
For example, if in the representation (\ref{p28}) coefficients $\alpha$, $\beta$ and $\gamma$
of the monomials, which include $\i_1$, are not equal to zero, then the following equality hold:
\begin{equation*}
 \begin{split}
  p=(\i_1+\frac{\gamma\delta-\alpha\rho+\sigma\beta}{2\beta\gamma}\i_2+
  &\frac{\gamma\delta+\alpha\rho-\sigma\beta}{2\alpha\gamma}\i_3+\\
  +\frac{-\gamma\delta+\alpha\rho+\sigma\beta}{2\alpha\beta}\i_4)
  &(\alpha\i_2+\beta\i_3+\gamma\i_4).
 \end{split}
\end{equation*}
Similar formulas are valid in the case when monomials, 
which include generators $\i_2$, $\i_3$, $\i_4$ have non-zero coefficients. 
To this category belong automatically  homogeneous elements $2$-degree with $5$ and $6$ monomials.

Now consider homogeneous elements of $2$-degree consisting of four monomials
such that each generator is present exactly in two monomials:
\begin{equation}\label{p17}
p=\alpha\i_1\i_2+\beta\i_1\i_3+\rho\i_2\i_4+\sigma\i_3\i_4. 
\end{equation}
Under the assumption that $a_1=1$, the system (\ref{p9}) can be rewritten for them in the form:
\begin{equation}\label{p18}
\left\{
\begin{aligned}
b_2+a_2b_1=\alpha,\\
b_3+a_3b_1=\beta,\\
b_4+a_4b_1=0,\\
a_2b_3+a_3b_2=0,\\
a_2b_4+a_4b_2=\rho,\\
a_3b_4+a_4b_3=\sigma.
\end{aligned}
\right.
\end{equation}
Taking coefficients $a_3=a$ and $b_1=b$ for the parameters,
from the system $(\ref{p18})$ we find that $\beta-2ab\not=0$ and
\begin{gather*}
 a_2=\frac{-a\alpha}{\beta-2ab},\ \ a_4=\frac{\sigma}{\beta-2ab},\ \ b_2=\frac{\alpha(\beta-ab)}{\beta-2ab},\\ 
 b_3=\beta-ab,\ \ b_4=\frac{-b\sigma}{\beta-2ab}.
\end{gather*}
Substituting these values into the fifth equation of the system (\ref{p18}), we obtain the equation
\begin{equation}\label{p19}
\frac{\alpha\beta\sigma}{(\beta-2ab)^2}=\rho.
\end{equation}
This shows that for the consistency of the system (\ref{p18}) requires 
that the value $\alpha\beta\sigma$ and $\rho$
have the same sign, which is equivalent to the inequality
\begin{equation}\label{p20}
\alpha\beta\rho\sigma>0.
\end{equation}
If the condition (\ref{p20}) does not hold, then the system (\ref{p18}) is not consistent and
the element (\ref{p17}) is prime. 
If the condition (\ref{p20}) holds, then $\frac{\alpha\beta\sigma}{\rho}>0$
and from (\ref{p19}) we get another condition for the parameters $a$ and $b$:
\begin{equation}\label{p21}
\beta-2ab=\pm\sqrt{\frac{\alpha\beta\sigma}{\rho}}.
\end{equation}
Picking $a$ and $b$ that  satisfies (\ref{p21}), 
we get a decomposition of the element (\ref{p17}) 
into  nontrivial prime factors. 
Thus lemma is completely proved.
\end{proof}

Consider in $P_4(\i)$ an arbitrary element of the form: 
\begin{equation}\label{p1}
u=a_1\i_1+a_2\i_2+a_3\i_3+a_4\i_4,\ \ \ a_i\in K,
\end{equation}
for which among the coefficients $a_i$ are three non-zero.
%---------------------------------------- Лемма 2 -----------------------------
\begin{lem}\label{pre1}
For any homogeneous third degree element  
\begin{equation*}\label{p2}
 w=\alpha\i_1\i_2\i_3+\beta\i_1\i_2\i_4+\gamma\i_1\i_3\i_4+\delta\i_2\i_3\i_4,
 \end{equation*}
$\alpha, \beta, \gamma, \delta\in K$, there exist a homogeneous second degree element
\begin{equation*}\label{p3}
 v=b_1\i_1\i_2+b_2\i_1\i_3+b_3\i_1\i_4+b_4\i_2\i_3+b_5\i_2\i_4+b_6\i_3\i_4, 
\end{equation*}
$b_j\in K$, such that
\begin{equation}\label{p4}
w=uv.
\end{equation}
\end{lem}
\begin{proof}
Assume for definiteness that $a_1a_2a_3\not=0$.
Equating coefficients in Equality (\ref{p4}), we get the system:
\begin{equation}\label{p5}
\left\{
\begin{aligned}
a_1b_4+a_2b_2+a_3b_1=\alpha,\\
a_1b_5+a_2b_3+a_4b_1=\beta,\\
a_1b_6+a_3b_3+a_4b_2=\gamma,\\
a_2b_6+a_3b_5+a_4b_4=\delta.
\end{aligned}
\right.
\end{equation}
With respect to variables $b_1,\dots,b_6$ the matrix of this systems has the form:
\begin{equation}\label{p6}
A=\begin{pmatrix}
  a_3&a_2&0&a_1&0&0\\
  a_4&0&a_2&0&a_1&0\\
  0&a_4&a_3&0&0&a_1\\
  0&0&0&a_4&a_3&a_2
\end{pmatrix}.
\end{equation}
Its rank is equal to $4$, since the determinant of the matrix, 
composed of the first, third, fifth and sixth  
columns of the matrix $A$, is not $0$:
\begin{equation*}
\begin{vmatrix}
  a_3&0&0&0\\
  a_4&a_2&a_1&0\\
  0&a_3&0&a_1\\
  0&0&a_3&a_2
\end{vmatrix}
=-2a_1a_2a_3^2.
\end{equation*}
Thus, the system (\ref{p5}) is consistent.

Similarly, we can prove this proposition in the case when 
other three coefficients $a_i$ are not zero.
\end{proof}
%-------------------------------------------------------------------------------------

%--------------------------------- Теорема 3 -------------------------------------
Consider the general case.
\begin{thm}\label{t3}
Only the following elements are prime in  $P_4(\i)$:
\begin{enumerate}
\item elements of $1$-degree;
\item elements that can be written as $p=q+r$ after reduction to the standard form,
where $r$ --- an element of degree greater than $2$, and $q$ ---  an element of one of the following:
\begin{enumerate}
\item $\alpha\i_a\i_b+\beta\i_a\i_c+\gamma\i_b\i_d,\ \ \alpha, \beta, \gamma\not=0$;  
\item $\alpha\i_a\i_b+\beta\i_a\i_c+\gamma\i_b\i_d+\delta\i_c\i_d,\ \ \alpha\beta\gamma\delta<0$;
\end{enumerate}
where all the indexes $a,b,c,d$ are distinct and take value from  $1$ to $4$.
\end{enumerate}
\end{thm} 
\begin{proof}
From the degree definition  and the property (\ref{1_5}) it follows 
that all elements of first degree are prime.
Let $p\in P_4(\i)$ be an arbitrary element of degree greater than $1$.
It can be uniquely written in the form:
\begin{equation}\label{p7}
p=q+r+\theta\i_1\i_2\i_3\i_4,
\end{equation}
where $q$ is a homogeneous element of degree two, 
and $r$ is a homogeneous element of degree three.
Obviously, if $p$ is decomposable, then $q$ is  decomposable or zero.
From this and from Lemma \ref{pre0} it follows that 
indicated in the second item of the theorem elements are prime.

Let us show that all  other  not invertible elements are decomposable.
Let $q\not=0$ and $q=ts$ be 
its decomposition into prime homogeneous factors of first-degree.
Analyzing Lemmas \ref{pre0} and \ref{pre1}, it is easy to see that one can always choose
one of the factors $t$ or $s$ in the form (\ref{p1}).
Let this be $t$. 
Then, by Lemma \ref{pre1}, there is a homogeneous element $z$ of two-degree  such that $tz=r$.
Suppose now that  an element $\alpha\i_1$ is one of the summands in $t$.
Then it is easy to see that the product
\begin{equation}\label{f29}
p=t(s+z+\frac{\theta}{\alpha}\i_2\i_3\i_4)
\end{equation}
is one of the decompositions into prime factors of $1$-degree of the element $p$.
Similarly, if other generators are summands in $t$.
If $q=0$, then, again,  by Lemma \ref{pre1} the element $p$ can be written  in the form (\ref{f29})
only without the term $s$ in the second factor.
\end{proof}

Let us return once more to the solution of  Equation (\ref{1_3}) 
in the case, when $a$ and $b$ are  not invertible elements.
If $\deg a=\deg b$, then this equation has a solution only in the case, 
when $a$ and $b$ differ by an invertible factor,
for example, $a=\iota_1$, $b=\iota_1+\iota_1\iota_2=a(1+\iota_2)$.
If $\deg b>\deg a$, then this equation has a solution only 
in the case, when $b$ is not prime, and $a$ is one of its factors.
For example, if $a$ is an element of first degree, and $b$ is 
one of elements of second degree specified in  Theorem \ref{t3}, then
(\ref{1_3}) does not have a solution.

%---------------------- Заключение -------------------------------------
\section*{Conclusion}
In this work we have considered prime elements of  Pimenov algebras 
with a number of generators no more than four.

As for Pimenov algebras  with more than four generators, 
it is easy to see that if an element $d$ is prime in an algebra $P_m(\i)$
then it is prime also in any algebra $P_n(\i)$, where $n>m$.
Hence taking into account Theorem \ref{t3}, it follows that 
there exist prime elements of the degree $2$ in algebras $P_n(\i)$, $n>4$.
It would be interesting to consider the general case and find out 
what are the other possible degrees of simple elements of Pimenov algebras.

The obtained results may be useful to development 
of factorization algorithms  in  Pimenov algebras.

The author thanks N.A.Gromov for valuable discussions and his useful comments. 

The study is supported by Program of UD RAS, project No 12-P-1-1013.


\begin{thebibliography}{99}
\bibitem{pim2}
{\it Pimenov R.I.\/} {Unified axiomatics of spaces with the maximum group of motions} // 
Litovski Matematicheski Sbornik, 1965. V.5. P. 457-486 (in Russian).
\bibitem{grom}
{\itshape  Gromov N.A.\/} Contractions and Analytical Continuations of the 
Classical Groups: Unified Approach. Komi Science Centre, Syktyvkar, 1990. 220 p. (in Russian).
\bibitem{grom_manko}
{\it Gromov N.A., Man'ko V.I.\/} {The Jordan-Schwinger representations of Cayley-Klein groups. I. The orthogonal groups} // J. Math. Phys., 1990, v.31, N5, p. 1047-1053.
\bibitem{grom_eng}
{\itshape Gromov N.A., Kostyakov I.V., Kuratov V.V.\/}  {Quantum
orthogonal    Cayley-Klein    groups    in     Cartesian     basis} //
Int.J.Mod.Phys.A., 1997, V.12, No.1, P. 33--41. q-alg/9610011.
\bibitem{grom2}
{\itshape  Gromov N.A.\/} {Possible quantum kinematics. II. Nonminimal case} //
Journal of Mathematical Physics.  2010. v. 51. N 8. 083515-1-12; arXiv:1001.3978v1. 
\bibitem{diment}
{\it Dimentberg F.M.\/} The screw calculus and its applications in mechanics. Foreign Technology Division, Wright-Patterson Air Force Base, Ohio, 1968. 155 p.
\bibitem{dupl}
{\it Duplij S.\/} {Nilpotent mehanics and supersymmetry} // 
Problems of Nuclear Physics and Cosmic Rays,
V. 30. Kharkov University Press, 1988. P. 41-48 (In Russian).
\bibitem{kisil}
{\it Kisil V.V.\/} {Induced Representations of the Group $SL_2(\rm R)$ and
Hypercomplex Numbers} // Proceedings of the Komi Science Centre UrD RAS. 1(5). 2011. P. 4-10 (in Russian).
\bibitem{sat2}
{\it Satarov Zh.S.\/} {Defining relations of the classical unitary group over the ring of dual numbers}
// Izv. Vyssh. Uchebn. Zaved. Mathematics, 1995, No.6, P. 74-81 (In Russian).
\bibitem{grassmann}
{\it Browne J.\/} Grassmann Algebra. Quantica Publishing, Melbourne, Australia, 2009.
\bibitem{witt} 
{\it DeWitt B.} Supermanifolds. Cambridge University Press, 1992.
\end{thebibliography}
\end{document}